\documentclass[11pt]{amsart}
\usepackage{amssymb,latexsym,url,amsmath}

\usepackage[T1]{fontenc}
\usepackage{hyperref}

\newcommand{\Fp}{{\mathbb{F}_p}}

\newcommand{\Char}{\operatorname{char}}

\newcommand{\Norm}{\operatorname{Norm}}

\newcommand{\disc}{\operatorname{disc}}

\newcommand{\Gal}{\operatorname{Gal}}

\newcommand{\K}{{\mathbb K}}

\newcommand{\isom}{ \cong }

\newcommand{\p}{\mathfrak{p}}

\newcommand{\Q}{{\mathbb Q}}

\newcommand{\Z}{{\mathbb Z}}
\newcommand{\ord}{\operatorname{ord}}
\newcommand{\Magma}{{\sf MAGMA }}

\newcommand{\Mathematica}{{\sf Mathematica }}
\newcommand{\FPP}{\operatorname{FPP}}
\newcommand{\f}{f_{d,c}}

\newcommand{\fo}{f_{d,c_0}}

\newcommand{\G}{G_{d,n}}
\newcommand{\Gu}{G_{d,u}}
\newcommand{\Gt}{G_{d,t}}

\renewcommand{\O}{\mathcal O}

\newcommand{\q}{\mathfrak q}

\newenvironment{Proof}{\par\noindent{\sc Proof:}}%
                      {\hspace*{\fill}\nobreak$\Box$\par\medskip}
                       {\hspace*{\fill}\nobreak$\Box$\par\medskip}

\newtheorem{Proposition}{Proposition}[section]
\newtheorem{Theorem}[Proposition]{Theorem}
\newtheorem{Lemma}[Proposition]{Lemma}

\newtheorem{Corollary}[Proposition]{Corollary}
\theoremstyle{definition}

\newtheorem{Remark}[Proposition]{Remark}
 \newtheorem{Example}[Proposition]{Example}
 \newtheorem{Question}[Proposition]{Question}
 
\usepackage[hmargin=1in,textheight=584pt]{geometry}

\renewcommand{\baselinestretch}{1.1}

\begin{document}

\title[Arbitrarily many primitive prime divisors]{Arbitrarily many primitive prime divisors in an iterate of $x^d+c$}

\author[M. Sadek]%
{Mohammad~Sadek}
\address{Mohammad Sadek: Faculty of Engineering and Natural Sciences, Sabanc{\i} University, Tuzla, \.{I}stanbul, 34956 Turkey}
\email{mohammad.sadek@sabanciuniv.edu}
\author[M. Wafik]%
{Mohamed~Wafik}
\address{Mohamed Wafik: Department of Mathematics, 1523 Greene Street, LeConte College, University of South Carolina, Columbia, SC 29208, USA}
\email{mwafik@sc.edu}

\date{}

\begin{abstract}
Let $K$ be a number field with ring of integers $R$, let $\f(x)=x^d+c\in K[x]$ be a unicritical polynomial of degree $d\ge2$, and let $\varrho_d(n,c)$ denote the number of primitive prime divisors of $\f^n(0)$, the $n$-th term of the critical orbit of $\f$. We show that $\varrho_d(n,c)$ admits no uniform bound, that is, no bound depending only on $d$ and $n$. More precisely, for all $n\ge1$ and $t\ge1$ there is $c\in R$ for which $\f^n(0)$ has at least $t$ primitive prime divisors, and these primes may be prescribed to occur with any given multiplicities. This can be done simultaneously for any finite set of iterates. The construction rests on a density computation. Namely, the set of primes occurring as a primitive prime divisor of $\f^n(0)$ for some $c$ has positive density. We compute that density explicitly for $d=2$ over $\Q$, under a hypothesis on the Galois group of the associated Gleason polynomial, verified for $n\le16$. As an application, we produce, for every $m\ge1$, a parameter $c\in R$ (in fact, a whole residue class of them), for which the splitting field of $\f^m(x)$ has the maximal possible degree over $K$.

\end{abstract}

\maketitle

\let\thefootnote\relax

\section{Introduction}

Let $R$ be an integral domain and $f(x)\in R[x]$. We write $f^n(x)$, $n\ge 0$, for the $n$-th iterate of $f(x)$. In other words, $f^0(x)=x$ and $f^n(x)=f(f^{n-1}(x))$. If $a_0\in R$, the orbit of $a_0$ under the iteration of $f$ is defined to be the set $$O_f(a_0)=\{a_0,f(a_0),\ldots, f^n(a_0),\ldots\}.$$
If $a_0$ is such that $f'(a_0)=0$, then $a_0$ is said to be a {\em critical point} of $f(x)$ and $O_f(a_0)$ is called a {\em critical orbit} of $f(x)$. 

If  $O_f(a_0)$ is a finite set, then $a_0$ is called a {\em preperiodic point of $f(x)$},  and $O_f(a_0)$ is called a {\em preperiodic orbit}. The latter is equivalent to the existence of two integers $m\ge 0$ and $n\ge 1$ such that $f^n(f^m(a_0))=f^m(a_0)$. If $n$ is the smallest such positive integer, then $n$ is called the {\em exact period} of $f^m(a_0)$, whereas $m$ is called the tail length of the orbit $O_f(a_0)$. In the latter case, $a_0$ is said to be a preperiodic point of $f(x)$ of {\em period type} $(m,n)$.  In case $m=0$, the point $a_0$ is called a {\em periodic point of $f(x)$} and $O_f(a_0)$ is called a periodic orbit. If $m\ge 1$, then $a_0$ is called a {\em strictly preperiodic point of $f(x)$}, and $O_f(a_0)$ is a strictly preperiodic orbit. 
If the orbit of $a_0$ is infinite, then $a_0$ is said to be a {\em wandering point} for $f(x)$. If all the critical orbits of $f(x)$ are finite, then $f(x)$ is said to be a {\em post-critically finite} polynomial, or a PCF polynomial, for short.  

Fix an integer $d\ge2$ and write $\f(x)=x^d+c\in \Q[x]$, a unicritical polynomial whose critical point is $0$. A prime $p$ is a \emph{primitive prime divisor} of $\f^n(0)$ if it divides $\f^n(0)$ and divides no earlier term of the critical orbit. In general, over a number field $K$ with ring of integers $R$, a prime $\p$ is said to be a primitive divisor of $\f^n(0)$ if $\nu_\p(\f^n(0))>0$ and $\nu_\p(\f^t(0))=0$ for $1\leq t<n$. Whether such primes exist is by now well understood. Doerksen and Haensch \cite{DoerksenHaensch+2012+465+472} settled the case over $\Z$, and Krieger \cite{Holly_Krieger} proved that for any $c\in\Q$ with infinite critical orbit, $\f^n(0)$ has a primitive prime divisor for every $n$ outside an explicit finite set. One may ask what proportion of all primes these primitive prime divisors constitute. Recall that the density of a
set $S$ of primes is defined by  $$\lim_{x\to\infty}\frac{\#\{\p\in S:\ \Norm_{K/\Q}(\p)\le x\}}{\#\{\p:\Norm_{K/\Q}(\p)\le x\}}, \quad\textrm{if the limit exists}.$$
It was proved in \cite[Theorem 2]{Odoni} that for the polynomial $f(x)=x^2-x+1\in\Q[x]$, the density of primes dividing at least one nonzero element in the orbit of $a_0\in \Z$ is $0$. 
The same density $0$ result holds for other families of polynomials including $f(x)=x^2-kx+k,\ k\in\Z,$ and $x^2+k,\, {k\in\Z\setminus\{-1\}}$, see \cite[Theorem 1.2]{Prim_Galois}. In addition,
under certain conditions on the field $K$ and the polynomial $\f(x)\in K[x]$, it was shown that the density of prime divisors in the orbit of any $a_0\in K$ under $\f(x)$ is zero, see \cite[Theorem 1.1]{Rafe_jones}.

Those results fix $c$ and let $n$ grow. This paper does the opposite. Fixing $n$
and letting $c$ vary, we write $\varrho_d(n,c)$ for the number of primitive prime divisors of $\f^n(0)$. An elementary argument bounds $\varrho_d(n,c)$ in terms of $d$, $n$ and the height of $c$ for $c\in\Q$. This bound is recorded in Theorem~\ref{bounding theorem}. The dependency on $d$ and $n$ is clear, but the dependency on $c$ is the one we question. We ask if we can get rid of this dependency and the answer is no. In fact, it fails as badly as it can.
This is Theorem~\ref{last theorem}, stated below as Theorem~\ref{first_theorem}
in the case $K=\Q$.

\begin{Theorem}
\label{first_theorem}
	Fix integers $d\ge 2$, $m\ge 1$ and let $P$ be a set of primes of density zero in the set of all primes. For $1\le i\le m$, let
	\begin{enumerate}
	    \item $n_i$ be distinct positive integers,
	    \item $t_i$ be positive integers,
	    \item $(k_{i,1},k_{i,2},\dots,k_{i,t_i})$ be $t_i$-tuples of positive integers.
	\end{enumerate}
	Then, there is a set of integers $c$ of positive density such that
	for each $1\leq i\leq m$ and $1\leq j\leq t_i$, there is a primitive prime divisor $p_{i,j}$ of $f_{d,c}^{n_i}(0)$ with $\nu_{p_{i,j}}(f_{d,c}^{n_i}(0))={k_{i,j}}$ and $p_{i,j}\not\in P$. 
\end{Theorem}

Theorem~\ref{first_theorem} is constructive, and the construction has a consequence for Galois groups. Hamblen, Jones and Madhu \cite{Rafe_jones} showed that, under an irreducibility hypothesis, a primitive prime divisor not dividing $d$ and occurring with multiplicity prime to $d$ forces the $n$-th layer of the arboreal extension attached to $\f$ to be as large as possible. Feeding the $k_j=1$ case of Theorem~\ref{first_theorem} into that criterion at every level up to $m$ gives the
following.
\begin{Theorem}[Corollary \ref{last_last_corollary} with $K=\Q$]
	\label{first_first_corollary}
	    Let $m\geq 1$, $d\geq 2$ and let $\phi$ denote the Euler-totient function.
        There is a set of integers $c$ of positive density such that the splitting field $F_m$ of $f_{d,c}^m(x)$ has Galois group $\Gal(F_{m}/\Q)$ of order $\phi(d)d^{\frac{d^{m}-1}{d-1}}$. Moreover, this is the maximal attainable order for $\Gal(F_{m}/\Q)$.
	\end{Theorem}

Let $\f(x)\in R[x]$ where $R$ is the ring of integers of a discrete valuation field and $k$ is the residue field of $R$ modulo the unique maximal ideal. We obtain information about the critical orbit of $\f(x)$ by investigating the structure of the critical orbit of the reduction of $\f(x)$ over the residue field. If $\f(x)$ possesses a strictly preperiodic point $r$ whose reduction has a unit multiplier, we prove that the reduction of $r$ is a strictly preperiodic point for the reduction of $\f(x)$ whose orbit has the same tail length as the tail length of $O_f(r)$. As for the exact period of $r$, it can be determined in terms of the exact period of its reduction, see \cite[Theorem 2.21]{dynamicsbook}.  On the other hand, if the characteristic of the residue field is larger than $d$ and $0$ is a periodic point of the reduction of $\f(x)$ over the residue field $k$, we show that $0$ is either wandering, or periodic of the same exact period over the ring $R$.

The latter results allow us to establish a connection between PCF polynomials of the form $\f(x)\in R[x]$, and PCF polynomials $\f(x)\in k[x]$. In fact, fixing $d\ge 2$ and assuming certain non-divisibility conditions on the discriminant of the associated Gleason polynomial, we show that there is a one-to-one correspondence between the polynomials $\f(x)\in R[x]$ with periodic critical orbits and the polynomials $\f(x)\in k[x]$ with periodic critical orbits. The one-to-one correspondence is explicit and it gives rise to a simple algorithm to find all polynomials $\f(x)\in R[x]$ with periodic critical orbits. The interested reader may consult \cite{Gleason} for the definition and properties of Gleason polynomials.
  
  Section~\ref{sec1} collects the local facts about critical orbits of $\f$ over a
discrete valuation ring that the later sections use.
Section~\ref{sec2} proves the lifting result. Section~\ref{pcf_zp_fp} introduces conditional correspondences between PCF polynomials over a local ring and its residue field. Section~\ref{sec:bounds} gives the
elementary upper bound and states Question~\ref{question4}. Section~\ref{link_sec}
computes the density. Section~\ref{sec:unbounded} proves Theorem~\ref{first_theorem} which is then used in Section~\ref{last_sec} to prove
Theorem~\ref{first_first_corollary}.
\subsection*{Acknowledgments} This work was initiated when M. Wafik was a master’s
student at Sabancı University under the supervision of M. Sadek. M. Sadek is supported by The Scientific and Technological Research Council of T\"{u}rkiye, T\"{U}B\.{I}TAK; research grant: ARDEB 1001/124F352. Claude (Anthropic, Opus 5.5; https://claude.ai) was used in September and October 2026 to review the manuscript and to draft parts of it. It was also used to write code for computational verification, in particular for the extended computation in Remark~\ref{claude_remark} and for the computational part of Theorem~\ref{deg_2_density}. The authors checked all suggested text, computations and references, and take full responsibility for the content.
\section{Preperiodic orbits of $\f(x)$ over local fields}\label{sec1}
 
In this section, $K$ is a discrete valuation field with ring of integers $R$, and discrete valuation $\nu$ corresponding to the unique maximal ideal $\p$ in $R$. The residue field of $R$ with respect to $\p$ is denoted by $k=R/\p$ with characteristic $p$. 
If $r\in R$, then we write $\widetilde{r}$ for its image under the reduction map $R\to k$, while in general we write $r+\p^t$ for the reduction of $r$ modulo $\p^t$, $t\ge1$. Similarly, the image of $f(x)\in R[x]$ under the reduction map $R[x]\rightarrow k[x]$ is denoted by $\widetilde{f}(x)$, and we write $f(x)+\p^t[x]$ for the reduction of $f(x)$ modulo $\p^t[x]$. The units of $R$ are denoted by $R^*$.

We set $\f(x)=x^d+c\in R[x]$, $d\geq 2$. In this section, we study the connection between the orbit structure of a point $r\in R$ under the iterations of $\f(x)\in R[x]$, and the orbit structure of $\widetilde{r}\in k$ under the iterations of $\widetilde{\f}(x)\in k[x]$. Since $k$ is a finite field, the orbit of $\widetilde{r}$ must be finite, hence $\widetilde{r}$ is preperiodic. First, we study the case when $\widetilde{r}$ is strictly preperiodic, i.e., the tail length of the orbit of $\widetilde{r}$ under the iteration of $\widetilde{\f}$ is not $0$. We then concern ourselves with the case $\widetilde{r}=0$ is periodic under $\widetilde{\f}(x)$. 

\subsection{Strictly preperiodic orbits}
 We recall that if $x_0$ is a periodic point of $f(x)\in K[x]$ with exact period $n$, the multiplier of $x_0$ is defined to be $\lambda_{x_0}(f):=\frac{\partial f^n(x)}{\partial x}|_{x=x_0}$.
 
Assuming that the orbit of $\widetilde{r}$ under $\widetilde{\f}(x)$ is strictly preperiodic in $k$, we are particularly interested in finding conditions under which the orbit of $r$ under $\f(x)$ is finite.

	\begin{Theorem}
	\label{last_preper}
		Let $\f(x)=x^d+c\in R[x]$ where $\nu\left(d\right)=0$. Let $r\in R$ be such that $\widetilde{r}$ is a strictly preperiodic point of $\widetilde{\f}(x)\in k[x]$ with period type $(m,n)$, and $\nu(\lambda)=0$ where $\lambda=\frac{\partial \f^{n}(x)}{\partial x}|_{x=\f^m(r)}$. Then either $r$ is a wandering point; or $r$ is preperiodic of type $(m,l)$ where $l=n$, $l=ns$ or $l=nsp^e$ with $s=\ord_{ k^*}(\widetilde{\lambda})$ and $e\ge 0$.
	\end{Theorem}
	\begin{Proof}
    \noindent\textbf{Step 1: Claim:} If $(m_t,n_t)$ is the period type of $r+\p^t$ in $ R/(\p^t)$, $t\ge 1$, define $g_t: \p^t/\p^{t+1}\to \p^t/\p^{t+1}$ by $g_t(y) =(\f^{n_t}(r_t+y)-r_t)+\p^{t+1}$ where $r_t=\f^{m_t}(r)$ and ${\f^{n_t}(r_t)=r_t+b}$ for some $b\in\p^t$. We claim that $g_t(y)=\lambda_t y + b+\p^{t+1}$ with $\lambda_t=\frac{\partial \f^{n_t}(x)}{\partial x}|_{x=\f^{m_t}(r)}$.
        
		\noindent\textbf{Proof of claim:} Write ${\f^{n_t}(x)=\underset{i=0}{\overset{d^{n_t}}{\sum}}a_i \cdot x^i}$ with $a_{d^{n_t}}=1$.		
		Note that $r_t$ is a fixed point for $\f^{n_t}(x)+\p^{t}[x]$ in $ R/(\p^t)$. In particular, $\underset{i=0}{\overset{d^{n_t}}{\sum}}a_i \cdot r_t^i=r_t+ b$ for some $b\in\p^t$. 
		We notice that $\lambda_t=\underset{i=1}{\overset{d^{n_t}}{\sum}}a_i\cdot i \cdot  r_t^{i-1}$. It follows that we get the following identity in $ R/(\p^{t+1})$ 
		\[ (\f^{n_t}(r_t+ y)-r_t)+\p^{t+1}=\underset{i=0}{\overset{d^{n_t}}{\sum}}a_i \cdot  (r_t+y)^i-r_t+\p^{t+1}.\]
		Since $y^2\in\p^{2t}\subseteq\p^{t+1}$, this yields that $(r_t+y)^i+\p^{t+1}=r_t^i+i\cdot r_t^{i-1}\cdot y+\p^{t+1}$. Thus,
		\begin{align*}
		    (\f^{n_t}(r_t+y)-r_t)+\p^{t+1}&=\underset{i=0}{\overset{d^{n_t}}{\sum}}a_i \cdot  (r_t^i+i\cdot r_t^{i-1}\cdot y)-r_t+\p^{t+1}\\
            &=\left(\underset{i=0}{\overset{d^{n_t}}{\sum}}a_i \cdot  r_t^i-r_t\right)+\underset{i=1}{\overset{d^{n_t}}{\sum}}a_i \cdot  (i\cdot r_t^{i-1}\cdot y)+\p^{t+1}\\
		    &=b +\left( \underset{i=1}{\overset{d^{n_t}}{\sum}}a_i \cdot  i\cdot r_t^{i-1}\right) y+\p^{t+1}.
		\end{align*}

    \noindent\textbf{Step 2: Claim:} $g_t^\alpha(0)=(\f^{\alpha\cdot n_t+m_t}(r)-\f^{m_t}(r))+\p^{t+1}$ for any $\alpha\geq 1$.
		
		\noindent\textbf{Proof of claim:} For $\alpha=1$, the claim is clear. Assume the claim holds for all $\alpha\le \alpha_0$. Now,
		\begin{eqnarray*}
		g_t^{\alpha_0+1}(0)&=&g_t(g_t^{\alpha_0}(0))=g_t\left(\f^{\alpha_0\cdot n_t+m_t}(r)-\f^{m_t}(r)\right)+\p^{t+1}\\
		&=&\left(\f^{n_t}\left(\f^{m_t}(r)+\left(\f^{\alpha_0\cdot n_t+{m_t}}(r)-\f^{m_t}(r)\right)\right)-\f^{m_t}(r)\right)+\p^{t+1}\\
		&=&(\f^{n_t}(\f^{\alpha_0\cdot n_t+{m_t}}(r))-\f^{m_t}(r))+\p^{t+1}\\
		&=&\left(\f^{(\alpha_0+1)\cdot n_t+{m_t}}(r)-\f^{m_t}(r)\right)+\p^{t+1}.
		\end{eqnarray*}

    \noindent\textbf{Step 3: Claim:} The tail length never grows. That is, for all $t\geq 1$, $r+\p^t$ is strictly preperiodic with period type $(m,n_t)$ for some $n_t\geq 1$.

    \noindent\textbf{Proof of claim:} We use an induction argument. 
		For $t=1$, the period type of $\widetilde{r}$ under $\widetilde{\f}(x)$ is $(m,n)$. Assume that $r+\p^{t}$ is strictly preperiodic for $\f(x)+\p^{t}[x]$ with period type $(m,n_t)$. 
		Since $ R/\p^{t+1}$ is finite, $r+\p^{t+1}$ will remain preperiodic, say with period type $(m_{t+1},n_{t+1})$. We want to show that the tail length $m_{t+1}=m$.
		Since the reduction of a periodic point is periodic, $m_{t+1}\geq m$. It suffices to show that $\f^m(r)+\p^{t+1}$ is periodic modulo $ \p^{t+1}$.
		
Setting		
		$r_t=\f^m(r),\ \f^{n_t}(r_t)=r_t+b$ where $b\in\p^t$,$\ g_t(y)=(\f^{n_t}(r_t+y)-r_t)+\p^{t+1}=\lambda_t y + b+\p^{t+1}$, one gets
		$\O_{g_t}(0)=\{0, b+\p^{t+1}, b(1+\lambda_t)+\p^{t+1},b(1+\lambda_t+\lambda_t^2)+\p^{t+1},\dots\}$. 

Since $m_t = m$, we have $n \mid n_t$ as $\widetilde{\f^{n_t+m}(r)}=\widetilde{\f^m(r)}$, and the chain rule gives $\widetilde{\lambda_t} = \widetilde\lambda^{\,n_t/n}$. Since $\nu(\lambda)=0$, we have $\lambda_t \in R^*$. If $\widetilde\lambda_t=1$, then $1+\lambda_t+\dots+\lambda_t^{p-1}\equiv p\equiv 0$ in $k$, where $p=\Char k$, and since $b\in\p^t$ it follows that $g_t^{\,p}(0)\in\p^{t+1}$. If $\widetilde\lambda_t\neq1$, set $s_t=\ord_{k^*}(\widetilde\lambda_t)$; then $1-\lambda_t$ is a unit, $\lambda_t^{s_t}\equiv1$ modulo $\p$, and
\[g_t^{\,s_t}(0)=b\left(1+\lambda_t+\dots+\lambda_t^{s_t-1}\right)+\p^{t+1}=b\cdot\frac{1-\lambda_t^{s_t}}{1-\lambda_t}+\p^{t+1}=0+\p^{t+1}.\]
Either way $0$ is a periodic point of $g_t$.

		It follows that if $\alpha$ is the period of $0$ under the iteration of $g_t$, then
		\[(\f^{\alpha\cdot n_t+m}(r)-\f^m(r))+\p^{t+1}=\p^{t+1}.\]
In other words,
		$\f^m(r)+\p^{t+1}\text{ is periodic for } \f(x)+\p^{t+1}[x].$

    \noindent\textbf{Step 4: Claim:} If $r$ is not wandering, then $r$ is preperiodic with period type $(m,l)$ where $l=n$, $l=ns$ or $l=nsp^e$ with $s=\ord_{ k^*}(\widetilde\lambda)$ and $e\ge 0$.
    
	\noindent\textbf{Proof of claim:} Assume that $r$ is not a wandering point. Thus, $r$ is a preperiodic point of type $(m_R,l)$ for $\f(x)$. In particular, $\f^{m_R}(r)$ is periodic, and so $\widetilde{\f}^{m_R}(\widetilde{r})$ is periodic in $k$. 
	
	Since the reduction of a periodic point is periodic, it cannot be the case that $m_R<m$. Therefore, we assume that $m_R>m$. It follows that $\f^m(r)$ is not periodic in $R$. This implies that $0<\nu(\f^{m+nl}(r)-\f^{m}(r))=t_0<\infty$, where the first inequality is due to the fact that $\widetilde{r}$ has period type $(m,n)$. In particular, we have that $\f^{m+nl}(r)- \f^m(r)\not\in\p^{t_0+1}$ but $\f^{m_R+nl}(r)- \f^{m_R}(r)=0\in\p^{t_0+1}$. It follows that $\f^m(r)+\p^{t_0+1}$ is not a periodic point for $\f(x)+\p^{t_0+1}[x]$. That in turn implies that the tail length of the orbit of $r+\p^{t_0+1}$ is strictly greater than $m$, contradicting Step 3. We conclude that $m_R=m$.
		
		If the period type of $r$ in $ R$ is $(m,l)$, then $\f^m(r)$ is a periodic point with period $l$ in $ R$ and its reduction $\widetilde{\f}^m(\widetilde{r})$ is periodic with period $n$ in $k$. By \cite[Theorem 2.21]{dynamicsbook}, we have that $l=n,\ \ l=ns,\text{ or }l=nsp^e$.
	\end{Proof}
	
	\begin{Corollary}
		Let $\f(x)=x^d+c\in\mathbb{Z}_p[x]$, where $p$ is a prime such that $p\nmid 6d$. If $\widetilde{r}$ is strictly preperiodic for $\widetilde{\f}(x)\in\Fp[x]$ with period type $(m,n)$ and multiplier $\lambda\in\Z_p^*$, then either $r$ is a wandering point; or $r$ is preperiodic of type $(m,l)$ where $l=n$ or $l=ns$ for $s=\ord_p(\widetilde\lambda)|(p-1)$.
	\end{Corollary}
	\begin{Proof}
		In the case of $K=\Q_p$ and $p\nmid6$, the possibility $l=nsp^{e}$ appearing in Theorem \ref{last_preper} does not occur, see \cite[Theorem 2.28]{dynamicsbook}. 
	\end{Proof}
	
	For $d=2$ and $r=0$, the latter corollary recovers \cite[Theorem 1.1]{Mullen}, which also treats $p=3$. Note that when $\widetilde{0}$ is strictly preperiodic, the condition $\nu(\lambda)=\nu(d^n\prod_{i=0}^{n-1}(\f^{m+i}(0))^{d-1})=0$ holds automatically.

\subsection{Preperiodic critical orbits}

In this subsection, we consider the critical orbit of $\f(x)$. If $\f(x)\in R[x]$ is such that $\widetilde{\f}(x)\in k[x]$ possesses a periodic critical orbit in $k$, we show that there are only two possibilities for the structure of the critical orbit of $\f(x)$ in $R$. Namely, that it is either infinite, or periodic with the same exact period as in the residue field. We start by recording the growth of the valuation through the orbit for the first case in the following lemma.

\begin{Lemma}
	\label{valuation lemma}
		Let $\f(x)=x^d+c\in R[x]$ where $\nu\left(d\right)=0$.
		Let $\widetilde{0}$ be a periodic point for $\widetilde{\f}(x)\in k[x]$ with exact period $n$. Then, for all $m\geq 1$, we have that
		\begin{enumerate}
		    \item \label{1} $\nu\left(\f^{mn+a}(0)-\f^{(m-1)n+a}(0)\right)=\nu\left(\f^{mn+1}(0)-\f^{(m-1)n+1}(0)\right)$,  where $n\ge a >0$.
		    \item\label{2} $\left(\f^{mn}(0)-\f^{(m-1)n}(0)\right)$ divides $\left(\f^{mn+1}(0)-\f^{(m-1)n+1}(0)\right)$.
		\end{enumerate}
	\end{Lemma}
	\begin{Proof}
	    Let $1\leq a\leq n-1$. Set $\beta_a:=\f^{mn+a}(0)-\f^{(m-1)n+a}(0)$. Since $\widetilde{0}$ is periodic with period $n$, it follows that $\beta_a\in\p$. 
		We also remark that since $n\nmid ((m-1)n+a)$, we have $\f^{(m-1)n+a}(0)\not\in\p$. 
        If $\beta_a\neq0$, let $\alpha\geq1$ be such that $\beta_a\in\p^\alpha\setminus\p^{\alpha+1}$.
		
		Now, we consider the following identities 
		\begin{align*}
		\beta_{a+1}+\p^{\alpha+1}&=
        \f^{mn+a+1}(0)-\f^{(m-1)n+a+1}(0)+\p^{\alpha+1}\\&=\left[\f^{mn+a}(0)\right]^d-\left[\f^{(m-1)n+a}(0)\right]^d+\p^{\alpha+1}\\
		&=\left[\f^{(m-1)n+a}(0)\right]^d+\beta_a\cdot d\cdot \left[\f^{(m-1)n+a}(0)\right]^{d-1}-\left[\f^{(m-1)n+a}(0)\right]^d+\p^{\alpha+1}\\
		&=\beta_a\cdot d\cdot \left[\f^{(m-1)n+a}(0)\right]^{d-1}+\p^{\alpha+1},
		\end{align*}
		where the third identity is due to the fact that $\left[\f^{mn+a}(0)\right]^d=\left[\f^{(m-1)n+a}(0)+\beta_a\right]^d$ and that $\beta_a^2\in\p^{\alpha+1}$.
		
		Since $d\cdot\f^{(m-1)n+a}(0)^{d-1}\not\in\p$ and $\beta_a\in\p^\alpha\setminus\p^{\alpha+1}$, it follows that 
		\[\f^{mn+a+1}(0)-\f^{(m-1)n+a+1}(0)\in\p^{\alpha}\setminus\p^{\alpha+1}\]
		which implies that 
		\[\nu(\beta_{a+1})=\nu(\beta_a) \quad\textrm{for any }\, 1\leq a\leq n-1.\]
        If $\beta_a=0$, then $\beta_{a+1}=[\f^{mn+a}(0)]^d-[\f^{(m-1)n+a}(0)]^d=0=\beta_a$.
        
        This concludes (\ref{1}) with $\nu(\beta_1)=\nu(\beta_2)=\dots=\nu(\beta_{n})$.

		For (\ref{2}), we write $\f^{(m-1)n}(0)=t_1$ and  $\left(\f^{mn}(0)-\f^{(m-1)n}(0)\right)=t_2$. Now, we have		
		\[\f^{mn+1}(0)-\f^{(m-1)n+1}(0)=\left[\f^{mn}(0)\right]^d-\left[\f^{(m-1)n}(0)\right]^d=(t_1+t_2)^d-t_1^d=t_2 \cdot  \underset{i=1}{\overset{d}{\sum}}\binom{d}{i}t_2^{i-1} t_1^{d-i}.\]
        This concludes the proof.
	\end{Proof}
	
\begin{Theorem}
		\label{last_periodic}
		Let $\f(x)=x^d+c\in R[x]$ where $\nu\left(d\right)=0$ and $\Char(k)=p>d$. If $\widetilde{0}$ is a periodic point for $\widetilde{\f}(x)$ in $k$ with exact period $n$, then either $0$ is a periodic point for $\f(x)$ with exact period $n$; or $0$ is a wandering point.
	\end{Theorem}
\begin{Proof}

\noindent\textbf{Step 1: Claim:} If 0 is not wandering, then its period type is $(m,n)$ for some $m\geq 0$.

\noindent\textbf{Proof of claim:} We assume that $0$ is not a wandering point. Thus, $\f^m(0)$ is a periodic point for some $m\ge 0$. We will show that the period of $\f^m(0)$ is exactly $n$. Assume that $\f^m(0)$ has period $l$. We choose an integer $u$ such that $u\geq \frac{m}{n}$. Then, $\f^{un}(0)$ is also periodic with exact period $l$.
		
		Since $\widetilde{0}$ is a periodic point for $\widetilde{\f}(x)\in k[x]$ with period $n$, then $\widetilde{\f}^{un}(\widetilde{0})=\widetilde{0}$, hence $\widetilde{\f}^{un}(\widetilde{0})$ has period $n$. In view of \cite[Theorem 2.21]{dynamicsbook}, we have $l=n$, $l=ns$ or $l=nsp^e$, where $s=\ord_{k^*}(\widetilde\lambda)$ and $\lambda=\frac{\partial \f^n(x)}{\partial x}|_{x=\f^{un}(0)}$. Since $\f^{un}(0)$ is a factor of the multiplier $\lambda$, it follows that $\lambda\in\p$. This means that $s=\infty$ contradicting the assumption that $0$ is not a wandering point. Therefore, $l=n$.

        \noindent\textbf{Step 2: Claim:} If 0 is strictly preperiodic, then the tail length is $mn+1$ for $m>0$.

        \noindent\textbf{Proof of claim:} Assume that $0$ is a strictly preperiodic point of $\f(x)$ with period type $(mn+a,n)$, for some $m\ge 0$ and $0<a\leq n$. So, we have that ${\f^{(m+1)n+a}(0)-\f^{mn+a}(0)=0}$. In view of Lemma \ref{valuation lemma} (\ref{1}), we must have 
		 $\f^{(m+1)n+1}(0)=\f^{mn+1}(0)$. In particular, $\f^{mn+1}(0)$ is a periodic point for $\f(x)$, and the tail length $mn+a$ is at most $mn+1$. Therefore, $a=1$.
		
		Now, we deal with the case $m=0$, i.e., we assume that $0$ is a preperiodic point of period type $(1,n)$. This means that $\f^{n+1}(0)=\f(0)=c$, or equivalently, $\left(\f^n(0)\right)^d+c=c$. This implies that $\f^n(0)=0$, which contradicts the assumption that $0$ is strictly preperiodic. Thus, $m\geq1$.

        \noindent\textbf{Step 3: Claim:} 0 is \textbf{not} strictly preperiodic.

        \noindent\textbf{Proof of claim:} We now show that the tail length $mn+1$ can never occur, so we assume on the contrary that the tail length is $mn+1$.
		
		We set $\f^{mn}(0)=t_1\ne 0$ with $\nu\left(t_1\right)=\alpha_1$, and $\f^{(m+1)n}(0)-\f^{mn}(0)=t_2\ne0$ with $\nu\left(t_2\right)=\alpha_2$.
Now, we have
		\begin{align*}
		\nu\left(\f^{(m+1)n+1}(0)-\f^{mn+1}(0)\right)&=\nu\left((\f^{(m+1)n}(0))^d-\f^{mn}(0)^d\right)=\nu\left((t_1+t_2)^d-t_1^d\right)\\
		&=\nu\left(\underset{i=1}{\overset{d}{\sum}}\binom{d}{i}t_2^i t_1^{d-i}\right).
		\end{align*}
		Since $p>d$, it follows that $p\nmid \binom{d}{i}$ for any $1\leq i\leq d$. This means that ${\nu\left(\binom{d}{i}t_2^it_1^{d-i}\right)=i\alpha_2+(d-i)\alpha_1}$. Assume $\alpha_1\neq\alpha_2$. If $\nu\left(\binom{d}{i}t_2^it_1^{d-i}\right)=\nu\left(\binom{d}{j}t_2^jt_1^{d-j}\right)$, then $i\alpha_2+(d-i)\alpha_1=j\alpha_2+(d-j)\alpha_1$ implying that $i=j$.
		Therefore, we obtain
		\[\nu\left(\underset{i=1}{\overset{d}{\sum}}\binom{d}{i}t_2^i t_1^{d-i}\right)=\underset{1\leq i\leq d}{\min}\left(\nu\left(\binom{d}{i}t_2^i t_1^{d-i}\right)\right)=\underset{1\leq i\leq d}{\min}\left(i\alpha_2+(d-i)\alpha_1\right)\neq \infty.\]
		It follows that $\f^{(m+1)n+1}(0)\ne\f^{mn+1}(0)$. 
		
		 Now, we are left with case $\alpha_1=\alpha_2=:\alpha$. We remark that $\f^{n}(\f^{mn}(0))=\f^{(m+1)n}(0)$, or equivalently, $\f^n(t_1)=t_1+t_2$. We write $\f^n(x)=\underset{i=0}{\overset{d^{n-1}}{\sum}}a_i \cdot  x^{di}\text{ with } a_0=\f^n(0)$,
		hence we obtain ${\f^n(t_1)=a_0+\underset{i=1}{\overset{d^{n-1}}{\sum}}a_i \cdot  t_1^{di}}$. In particular,
		$t_1+t_2-\underset{i=1}{\overset{d^{n-1}}{\sum}}a_i \cdot  t_1^{di}=a_0$, and hence
		 $\nu\left(a_0\right)=\nu\left(\f^n(0)\right)\geq \alpha$.
		
		If $\nu\left(\f^n(0)\right)=\infty$, then $0$ is a periodic point of $\f(x)$ which contradicts the assumption that $0$ is a preperiodic point of type $(mn+1,n)$.  Thus, $\nu\left(\f^n(0)\right)<\infty$. 
		
		We notice that
		\[\nu\left(\f^{n+1}(0)-\f(0)\right)=\nu\left(\f^{n}(0)^d\right)=d\cdot \nu\left(\f^{n}(0)\right)\geq d\alpha> \alpha.\]
		By Lemma \ref{valuation lemma}(\ref{1}), $\nu\left(\f^{jn+n}(0)-\f^{jn}(0)\right)=\nu\left(\f^{jn+1}(0)-\f^{(j-1)n+1}(0)\right)$. In particular, for $j=1$ we get that $\nu\left(\f^{2n}(0)-\f^{n}(0)\right)=\nu\left(\f^{n+1}(0)-\f(0)\right)>\alpha$.
		Also, in view of Lemma \ref{valuation lemma}(\ref{2}), we have that
		\[\nu\left(\f^{(j+1)n}(0)-\f^{jn}(0)\right)=\nu\left(\f^{jn+1}(0)-\f^{(j-1)n+1}(0)\right)\geq\nu\left(\f^{jn}(0)-\f^{(j-1)n}(0)\right).\]
		An induction argument shows that  
		$\nu\left(\f^{(m+1)n}(0)-\f^{mn}(0)\right)> \alpha.$
		This contradicts the definition of $\alpha$ as $\alpha=\nu\left(\f^{(m+1)n}(0)-\f^{mn}(0)\right)$.
	\end{Proof}
	
		For $K=\Q_p$ and $d=2$, Theorem \ref{last_periodic} reduces to \cite[Proposition 2.4]{Mullen}.

\section{Polynomial dynamical systems modulo prime powers}\label{sec2}

In this section, $K$ is a number field with ring of integers $R$. If $\p$ is a prime ideal in $R$, then $K_{\p}$ is the completion of $K$ with respect to $\p$, $R_{\p}$ is its ring of integers, $\nu_{\p}$ is  the discrete valuation corresponding to the place $\p$ of $K$, and $k_{\p}$ is the corresponding residue field. We write $N_{K/\Q}$ for the norm map.

We recall that a {\em critical point} of a polynomial $f$ is a point $s$ such that $f'(s)=0$.
 If all critical points of $f$ have finite orbits, then $f$ is called post-critically finite (PCF).  The polynomial $\f(x)$ has a unique critical point, namely $0$, and we refer to its orbit as the critical orbit of $\f(x)$.

We start this section by establishing a link between primitive prime divisors in the critical orbit of the polynomial $\f(x)\in K[x]$ and the periodicity of the critical orbit of the reduction of $\f(x)$ modulo these primes.

The following proposition is a direct consequence of the definitions of primitive prime divisors and the period of a periodic point as well as the fact that the critical orbit is a rigid divisibility sequence, see \cite[Lemma 2.5]{Rafe_jones} for more details.

\begin{Proposition}
\label{prop_prim_per}
Let $\f(x)=x^d+c\in K[x]$. The prime $\p$ is a primitive prime divisor for $\f^n(0)$, $n\ge 1$, with $\nu_{\p}\left(\f^n(0)\right)\geq t$, $t\geq1$, if and only if $0+\p^t$ is a periodic point for $\f(x)+\p^t[x]$ with exact period $n$.
\end{Proposition}

\begin{Theorem}
		\label{lifting theorem}
		Fix $d\ge 2$. Let $c_0\in K$ be such that $\p$ is a primitive prime divisor for $\fo^n(0)$.
		If $\nu_{\p}\left(\fo^n(0)\right)>2\nu_{\p}\left(\frac{\partial \f^n(0)}{\partial c}|_{c=c_0}\right)$, then there is a unique $\overline{c_0} \in R_{\p}$ such that the following statements hold:
		\begin{enumerate}
		    \item $\nu_{\p}\left(\overline{c_0}- c_0\right)>\nu_{\p}\left(\frac{\partial \f^n(0)}{\partial c}|_{c=c_0}\right)\geq 0$.
		    \item The point 0 is periodic with exact period $n$ for  $f_{d,\overline{c_0}}(x)\in R_{\p}[x]$.
		\end{enumerate}
		In particular, $\nu_{\p}\left(\overline{c_0}- c_0\right)=\nu_{\p}\left(\fo^n(0)\right)-\nu_{\p}\left(\frac{\partial \f^n(0)}{\partial c}|_{c=c_0}\right)$.
	\end{Theorem}
\begin{Proof}
		That $\p$ is a primitive prime divisor for $\fo^n(0)$ means that $\widetilde{\fo^n}(0)=\widetilde{0}$, see Proposition \ref{prop_prim_per}. Also, since $\nu_{\p}\left(\fo^n(0)\right)>2\nu_{\p}\left(\frac{\partial \f^n(0)}{\partial c}|_{c=c_0}\right)$, it follows by Hensel's lemma \cite[Theorem 9.1]{Conrad_Hensel} \cite[Theorem 6.28]{advanced_algebra}, that there is a unique $\overline{c_0} \in R_{\p}$ such that 
		${\nu_{\p}\left(\overline{c_0}- c_0\right)>\nu_{\p}\left(\frac{\partial \f^n(0)}{\partial c}|_{c=c_0}\right)}$ and $f_{d,\overline{c_0}}^n(0)=0\in R_{\p}$. We also have that $\nu_{\p}\left(\overline{c_0}- c_0\right)=\nu_{\p}\left(\fo^n(0)\right)-\nu_{\p}\left(\frac{\partial \f^n(0)}{\partial c}|_{c=c_0}\right)$.
		
		Since $\widetilde{\overline{c_0}}=\widetilde{c_0}$, it follows that $\widetilde{f_{d,\overline{c_0}}^t}(0)=\widetilde{f_{d,c_0}^t}(0)$ for all integers $t\geq1$. Moreover, since $\p$ is a primitive prime divisor for $\fo^n(0)$, we have that $\widetilde{f_{d,\overline{c_0}}}^t(0)=\widetilde{f_{d,c_0}}^t(0)\neq\widetilde{0}$, for any $t<n$. This implies that $f_{d,\overline{c_0}}^t(0)\neq0\in R_{\p}$, i.e., 0 has exact period $n$ for $f_{d,\overline{c_0}}(x)\in R_{\p}[x]$.
	\end{Proof}

\begin{Example}
	\label{lifting example}
	Let $K=\Q$. The prime $5$ is a primitive divisor for $f_{2,1}^3(0)$. We also have that
	$\nu_5(f_{2,1}^3(0))=1>0=2\cdot \nu_5\left(\frac{\partial f_{2,c}^3(0)}{\partial c}|_{c=1}\right).$
	By Theorem \ref{lifting theorem}, there is $\overline{c_0}\in\Z_5$ such that $\overline{c_0}\equiv 1 \mod 5$, and $f_{2,\overline{c_0}}^3(0)=0\in\Z_5$. In fact, we may show that
	$\overline{c_0}=1+5t_0$ for some $t_0\in\Z_5$ where
 $t_0=3+25t_1$ for some $t_1\in\Z_5$. In particular, $\overline{c_0}\equiv16\mod125$. 
\end{Example}

In the following example, we show that the conditions in Theorem \ref{lifting theorem} cannot be dropped. We can see that for any value close to $c_0$, i.e., ${c_0'}=c_0+p\cdot t$ where $t\in\Z_p$, we get $\nu_p(f_{2,{c_0'}}^n(0))\leq \nu_p(f_{2,c_0}^n(0))$. This means that $\nu_p(f_{2,{c_0'}}^n(0))<\infty$, and so, 0 is not periodic of period $n$ for $f_{2,{c_0'}}$.

\begin{Example}
\label{bad_ex}
	Let $K=\Q$. The prime $13$ is a primitive divisor for $f_{2,3}^5(0)$. We see that $\nu_{13}(f_{2,3}^5(0))=1$ and $\nu_{13}\left(\frac{\partial f_{2,c}^5(0)}{\partial c}|_{c=3}\right)=1$.
	It can be seen that for any $t_0\in\Z_{13}$, we have $f_{2,3+13t_0}^5(0)\not\equiv 0 \mod 13^2$. This implies that $f_{2,3+13t_0}^5(0)\neq 0\in\Z_{13}$. 
\end{Example}

Fix an integer $n\ge 1$. Given a polynomial $f(x)\in F[x]$, where $F$ is a field, we can define the corresponding {\em dynatomic polynomial} as follows \cite[Section 4.1]{dynamicsbook}
    \[\phi_{n,f}(x):=\underset{t|n}{\prod}(f^t(x)-x)^{\mu(\frac{n}{t})}\]
    where $\mu$ is the M\"{o}bius function.
    We note that the roots of the dynatomic polynomial are periodic points of $f(x)$ of period $n$ but not necessarily of exact period $n$. Therefore, the roots of the dynatomic polynomial are said to be points of {\em formal period} $n$. Now, the {\em Gleason polynomials} associated to $\f(x)\in F[x]$ are defined as follows
    \begin{equation*}
    G_{d,n}(c)=\phi_{n,\f}(0):=\underset{t|n}{\prod}(\f^t(0))^{\mu(\frac{n}{t})}.
    \end{equation*}
    We note that by M\"{o}bius inversion, we get that
    \[\f^n(0)=\underset{t|n}{\prod}G_{d,t}(c).\]

In what follows we replace the condition $\nu_{\p}\left(\fo^n(0)\right)>2\nu_{\p}\left(\frac{\partial \f^n(0)}{\partial c}|_{c=c_0}\right)$ in Theorem \ref{lifting theorem} by a condition involving Gleason polynomials. 

\begin{Lemma}
		\label{disc lemma}
		Let $\p$ be a primitive prime divisor for $\fo^n(0)$ for some $c_0\in K$. If $\nu_{\p}\left(\frac{\partial \f^n(0)}{\partial c}|_{c=c_0}\right)>0$, then $\disc_c(\G(c))\in\p$,
		where $\G(c):=\phi_{n,\f}(0)$ is the Gleason polynomial.
	\end{Lemma}
\begin{Proof}
		Recall that $\p$ being a primitive prime divisor for $\fo^n(0)$ is equivalent to $0$ being a  periodic point for $\widetilde{\fo}(x)$ in $k_{\p}$ with exact period $n$, see Proposition \ref{prop_prim_per}. Since the cycle through $0$ has multiplier $0$, which is not a root of unity, \cite[Theorem 4.5]{dynamicsbook} shows that $0$ has formal period $n$ if and only if it has exact period $n$. This implies that $\widetilde{c_0}$ is a root of the polynomial $\widetilde{\G}(c)$, i.e., $\widetilde{\G}(\widetilde{c_0})=0\in k_{\p}$.
	
		Now, the derivative of $\f^n(0)=\underset{t|n}{\prod}\Gt(c)$ satisfies the identity
		\[\frac{\partial \f^n(0)}{\partial c}=\underset{t|n}{\sum}\frac{\partial \Gt(c)}{\partial c} \underset{\underset{u\neq t}{u|n}}{\prod}\Gu(c).\]
		Since $\widetilde{\G}(\widetilde{c_0})=0$ and $\G(c_0)$ divides each term except when $t=n$, we have
		\[\frac{\partial \widetilde{\f}^n(0)}{\partial c}|_{c=\widetilde{c_0}} = \frac{\partial \widetilde{\G}(c)}{\partial c}|_{c=\widetilde{c_0}} \underset{\underset{u\neq n}{u|n}}{\prod}\widetilde{\Gu}(\widetilde{c_0})\]
		Since $\p$ is a primitive prime divisor for $\f^n(0)$, we have that $\widetilde{\Gu}(\widetilde{c_0})\neq0 \text{ for any }u<n$. Therefore,
		\[\frac{\partial \widetilde{\f}^n(0)}{\partial c}|_{c=\widetilde{c_0}}=0 \text{ if and only if } \frac{\partial \widetilde{\G}(c)}{\partial c}|_{c=\widetilde{c_0}}=0.\]
		Since $\widetilde{\G}(\widetilde{c_0})=0$, it follows that
		$\frac{\partial \widetilde{\f}^n(0)}{\partial c}|_{c=\widetilde{c_0}}=0$ if and only if $\widetilde{c_0}$ is a repeated root of $\widetilde{\G}(c)$. This implies that $\disc_c(\widetilde{\G}(c))=0$. This concludes the proof.
	\end{Proof}
\begin{Remark}
		\label{remark 1}
		The converse of Lemma \ref{disc lemma} does not hold in general. For example, $G_{2,3}(c)=c^3+2c^2+c+1\equiv(c+8)^2(c+9)\mod\ 23$. Therefore, $23|\disc_c(G_{2,3})$.  Also, $23$ is a primitive prime divisor for $f_{2,-9}^3(0)$. However, $23\nmid \frac{\partial f_{2,c}^3(0)}{\partial c}|_{c=-9}$.
	\end{Remark}
Lemma \ref{disc lemma} leads to a simplification of the condition in Theorem \ref{lifting theorem} as follows.
	\begin{Corollary}
		\label{lifting_corollary}
		Let $c_0\in K$ be such that $\p$ is a primitive prime divisor for $\fo^n(0)$.
		If $\disc_c(\G(c))\not\in\p$, then there is a unique $\overline{c_0} \in R_{\p}$ such that the following statements hold:
		\begin{enumerate}
		    \item $\nu_{\p}\left(\overline{c_0}- c_0\right)> 0$.
		    \item The point 0 is periodic with exact period $n$ for  $f_{d,\overline{c_0}}(x)\in R_{\p}[x]$.
		\end{enumerate}
		Furthermore, $\nu_{\p}\left(\overline{c_0}- c_0\right)=\nu_{\p}\left(\fo^n(0)\right)$.
	\end{Corollary}
\begin{Proof}
Since $\p$ is a primitive prime divisor for $\fo^n(0)$, we have $\nu_{\p}(\fo^n(0))>0$. According to Lemma \ref{disc lemma}, $\disc_c(\G(c))\not\in\p$ implies that $\nu_{\p}\left(\frac{\partial \f^n(0)}{\partial c}|_{c=c_0}\right)= 0$. Thus, the conditions of Theorem \ref{lifting theorem} are satisfied.
\end{Proof}
In the following corollary, we show the existence of certain values of $c$ such that a specific power of $p$ divides $\f^n(0)$. 
\begin{Corollary}
		\label{power_change}
		Let $c_0\in K$ be such that $\p$ is a primitive prime divisor for $\fo^n(0)$.
		If $\disc_c(\G(c))\not\in\p$, then for any integer $r\geq1$, there is $c_r\in R$ such that $\p$ is a primitive prime divisor for $f_{d,c_r}^n(0)$ and $\nu_{\p}(f_{d,c_r}^n(0))=r$.
\end{Corollary}
\begin{Proof}
Let $\overline{c_0}$ be as in Corollary \ref{lifting_corollary}, and $c_r\in R\subseteq R_{\p}$ be such that $\nu_{\p}(c_r-\overline{c_0})=r$. By the choice of $c_r$, we have that $\widetilde{c_r}=\widetilde{\overline{c_0}}=\widetilde{c_0}$. This means that $\p$ is a primitive prime divisor of $f_{d,c_r}^n(0)$, see Proposition \ref{prop_prim_per}. It follows that there is a unique $\overline{c_r}\in R_{\p}$ such that $\nu_{\p}(\overline{c_r}-c_r)>0$ and $f^n_{d,\overline{c_r}}(0)=0$, see Corollary \ref{lifting_corollary}. From the uniqueness of $\overline{c_r}$, we get that $\overline{c_r}=\overline{c_0}$. So,
\[r=\nu_{\p}(\overline{c_0}-c_r)=\nu_{\p}(\overline{c_r}-c_r)=\nu_{\p}(f_{d,c_r}^n(0)),\]
concluding the proof.
\end{Proof}

For an integer $r\geq 1$, Corollary \ref{power_change} implies that except for the finitely many primes dividing $\disc_c(G_{d,n}(c))$, once we know that a prime $\p$ can appear as a primitive prime divisor of $\fo^n(0)$ for some $c_0\in K$, there is another value $c_r\in R$ for which $\p$ is a primitive prime divisor of $f_{d,c_r}^n(0)$, and $\nu_{\p}(f_{d,c_r}^n(0))=r$.
\begin{Example}
The prime $5$ is a primitive divisor for $f_{2,1}^3(0)$ with $\nu_{5}(f_{2,1}^3(0))=1$ and $\nu_5\left(\frac{\partial f_{2,c}^3(0)}{\partial c}|_{c=1}\right)=0$. One can see that $5^2||f_{2,-9}^3(0)$.
\end{Example}

For the following theorem, we need to consider the density of a subset of the ring $R$ with respect to an integral basis, i.e., identifying $R$ with $\Z^{[K:\Q]}$.

\begin{Theorem}
	\label{collecting theorem}
		Fix $d\ge 2$. If there exist a finite set of $K$-rational numbers $c_i$ and distinct primes $\p_i$, $1\le i\le m$, such that $\p_i$ is a primitive prime divisor for $f_{d,c_i}^{n_i}(0)$ and $\nu_{\p_i}(f_{d,c_i}^{n_i}(0))=k_i$, for some $n_i,k_i\geq 1$,  
		then there exists a set of elements $\overline{c}\in R$ of positive density such that 
		\begin{enumerate}
		    \item $\p_i$ is a primitive prime divisor for $f_{d,\overline{c}}^{n_i}(0)$, $1\le i\le m$, and
		     \item $\nu_{\p_i}( f_{d,\overline{c}}^{n_i}(0))=k_i$, $1\le i\le m$.
		\end{enumerate}
\end{Theorem}
\begin{Proof}
This is a simple corollary to the Chinese Remainder Theorem \cite[Theorem 5.33]{Jarvis}. Taking an element $\overline{c}\in R$ such that $\overline{c}+\p_i^{k_i+1}=c_i+\p_i^{k_i+1}$, $1\le i\le m$, and noting that $\nu_{\p}(\f^n(0))=r$ where $\p$ is a primitive prime divisor of $\f^n(0)$ if and only if $0$ is periodic for $\f(x)+\p^r[x]$ but $0$ is not periodic for $\f(x)+\p^{r+1}[x]$ yield the proof. Since this is true for any representative of $\overline{c}+\prod_i\p_i^{k_i+1}$ (as the proof only depends on the reduction modulo $\p_i^{k_i+1}$), the positive density holds.
\end{Proof}

\section{Correspondence between PCF polynomials in $k_{\p}[x]$ and $R_{\p}[x]$} 
\label{pcf_zp_fp}

  Following the notation of the previous section, we use Theorem \ref{lifting theorem} to give conditions that allow a one-to-one correspondence between polynomials $\f(x)\in k_{\p}[x]$ with periodic critical orbits and polynomials $\f(x)\in R_{\p}[x]$ with periodic critical orbits.

  \begin{Theorem}
      \label{one_cor}
      Fix $d\ge 2$. Let $\p$ be a prime ideal with $\Char(k_{\p})>d$, and $n$ be a positive integer. If $1\leq n\leq N_{K/\Q}(\p)$ and $\disc_c(\G(c))\not\in\p$, then there is a one-to-one correspondence
		\[
		A_n:=\left\{\begin{array}{cc}
			
			\text{$\f(x)\in k_{\p}[x]$ with }
			\\
			\text{periodic critical orbit}
			\\
			\text{of exact period $n$}
		\end{array}\right\}
		\longleftrightarrow
		B_n:=\left\{\begin{array}{cc}
			\text{$\f(x)\in R_{\p}[x]$ with }
			\\
			\text{periodic critical orbit}
			\\
			\text{of exact period $n$}
		\end{array}\right\}\]
		Moreover, for $n>N_{K/\Q}(\p)$, $A_n=B_n=\emptyset$.
  \end{Theorem}
  
  \begin{Remark}
	\label{1-1-1-1}
    The condition $\disc_c(G_{d,n}(c))\not\in\p$ can be weakened to allow repeated irreducible non-linear factors of $\widetilde{G_{d,n}(c)}$ over $k_\p$ by using the following condition:
    \begin{align}
    \label{condition_star}
    &\textbf{For all }{c\in k_{\p}}\textbf{, if 0 is periodic with exact period }n\textbf{ for }{\f(x)\in k_{\p}[x]}\textbf{, then}\nonumber\\
    &c\textbf{ is a simple root of }{\f^n(0)\in k_{\p}[c]}.\tag{$*$}
    \end{align}
\end{Remark}
\begin{Proof}
First, we note that for $1\leq n\leq N_{K/\Q}(\p)$, if $\disc_c(\G(c))\not\in\p$, then condition (\ref{condition_star}) holds by Lemma \ref{disc lemma}. Let $\{\alpha_i+\p\}_{1\leq i\leq N_{K/\Q}(\p)}$ be the set of cosets of $\p$ in $R$. We define the map $\iota: k_{\p}\to R$ by $\iota(c_0)=\alpha_i\in R$ for some $1\leq i\leq N_{K/\Q}(\p)$ such that $\widetilde{\alpha_i}=c_0$. 

The condition (\ref{condition_star}) implies that if 0 has exact period $n$ for $\fo$ for some $c_0\in k_\p$, then $\frac{\partial \f^n(0)}{\partial c}|_{c=c_0}\neq 0$. The latter means that if ${\p}$ is a primitive prime divisor of $f_{d,\iota(c_0)}^n(0)$ then $\nu_{\p}\left(\frac{\partial \f^n(0)}{\partial c}|_{c=\iota(c_0)}\right)=0$. In  view of Theorem \ref{lifting theorem}, there is a unique $c_1\in R_{\p}$ such that $\nu_{\p}(c_1-\iota(c_0))>0$ and 0 is periodic for $f_{d,c_1}(x)\in R_{\p}[x]$ with exact period $n$. The condition that $\nu_{\p}(c_1-\iota(c_0))>0$ implies that $\widetilde{c_1}=\widetilde{\iota(c_0)}=c_0$. With this, we define the map $\psi:A_n\to B_n$, where $\psi(\fo(x))=f_{d,c_1}(x)$ with $c_0$ and $c_1$ as defined above.

Let $c_0\in k_{\p}$ be such that $\fo\in A_n$. Assume that $\psi(\fo(x))=f_{d,c_i}(x)$, $i=1,2$. Then $\widetilde{c_1}=c_0=\widetilde{c_2}$ and 0 is periodic with exact period $n$ for both $f_{d,c_1}(x)$ and $f_{d,c_2}(x)$. By the uniqueness given in Theorem \ref{lifting theorem}, we get that $c_1=c_2$. So, $\psi$ is a well defined map.

Let $\fo(x), f_{d,c_3}(x)\in A_n$ with $\psi(\fo(x))=\psi(f_{d,c_3}(x))=f_{d,c_2}(x)$. Then we get that $c_0=\widetilde{c_2}=c_3$, i.e., $\psi$ is injective.

Let $f_{d,c_1}(x)\in B_n$. Then $0$ is periodic with exact period $m$ for $\widetilde{f_{d,c_1}}(x)=f_{d,\widetilde{c_1}}(x)$, where $m|n$. By Theorem \ref{last_periodic}, since $\Char(k_{\p})>d$, we have that $m=n$, i.e., $f_{d,\widetilde{c_1}}(x)\in A_n$. By the uniqueness in Theorem \ref{lifting theorem}, $\psi(f_{d,\widetilde{c_1}}(x))=f_{d,c_1}(x)$. So, $\psi$ is surjective.

For the case that $n>N_{K/\Q}({\p})$, if $f_{d,c_1}(x)\in B_n$, then again by Theorem \ref{last_periodic}, $\widetilde{f_{d,c_1}}(x)\in A_n$. However, the orbit length in $k_{\p}$ cannot be greater than $N_{K/\Q}({\p})$. Thus, $B_n$ is empty.
\end{Proof}

\begin{Corollary}
		\label{one-to-one}
	Fix $d\ge 2$. Let $\p$ be a prime ideal with $\Char(k_{\p})>d$ such that $\disc_c(\G(c))\not\in\p$ for every $n\leq N_{K/\Q}(\p)$. Then, there is a one-to-one correspondence
		\[
		\left\{\begin{array}{cc}
			
			\text{$\f(x)\in k_{\p}[x]$ with }
			\\
			\text{periodic critical orbit}
		\end{array}\right\}
		\longleftrightarrow
		\left\{\begin{array}{cc}
			\text{$\f(x)\in R_{\p}[x]$ with }
			\\
			\text{periodic critical orbit}
		\end{array}\right\}\]
	\end{Corollary}

    \begin{Remark}
        \label{1-1-1}
        The condition $\disc_c(G_{d,n}(c))\not\in\p$ for any $n\leq N_{K/\Q}(\p)$ can be again weakened as follows:
    \begin{align}
    \label{condition_star_star}
    &\textbf{For all }{c\in k_{\p}}\textbf{, if 0 is periodic with exact period }n\textbf{ for }{\f(x)\in k_{\p}[x]}\textbf{ for}\nonumber\\
    &\textbf{some }{n\geq1}\textbf{, then }c\textbf{ is a simple root of }{\f^n(0)\in k_{\p}[c]}.\tag{$**$}
    \end{align}
    \end{Remark}
	
	\begin{Proof}
	    We begin by noting that for all $1\leq n\leq N_{K/\Q}(\p)$, if $\disc_c(\G(c))\not\in\p$, then condition (\ref{condition_star_star}) follows, see Lemma \ref{disc lemma}. The rest is a direct corollary of Theorem \ref{one_cor} and Remark \ref{1-1-1-1}. Since for each $n\geq 1$, there is a one-to-one correspondence between $A_n$ and $B_n$, we get the one-to-one correspondence between the disjoint unions $\underset{1\leq n\leq N_{K/\Q}(\p)}{\cup} A_n$ and $\underset{1\leq n\leq N_{K/\Q}(\p)}{\cup} B_n$. Moreover, since Theorem \ref{one_cor} shows that $A_n=B_n=\emptyset$ for $n>N_{K/\Q}(\p)$, we are done.
	\end{Proof}

\begin{Remark}
\label{claude_remark}
    For $d=2$ and $K=\Q$, condition (\ref{condition_star_star}) was tested for the first 50 odd primes using $\Mathematica$ \cite{Mathematica}. It holds for all but four of them, the exceptions being $p=13,\ 23,\ 137$ and $211$. The case $p=13$ is Example \ref{bad_ex}. We extended the check up to $10^5$ and found only 80 failing primes out of 9591 odd primes below $10^5$.

\end{Remark}

	Using Corollary \ref{one-to-one}, if $\p$ with $\Char(k_{\p})>d$ satisfies the given condition, then after finitely many steps one may find all polynomials $\f(x)\in R_{\p}[x]$ with periodic critical orbits.

	We conclude this section by an example illustrating the findings of Corollary \ref{one-to-one}.
	\begin{Example}
		For $d=3$, $K=\Q$ and $p=5$,  one notices that
		\[\disc_c(G_{3,i}(c))\equiv 1\not\equiv 0\mod 5, \ 1\le i\le 4,\quad\textrm{and }
		\disc_c(G_{3,5}(c))\equiv 4\not\equiv 0\mod 5.\]
		So, in order to find the polynomials of the form $f_{3,c}(x)\in\Z_5[x]$ such that $0$ has a periodic orbit, it suffices to find the polynomials $f_{3,c}(x)\in\mathbb{F}_5[x]$ with $0$ being a periodic point. In $\mathbb{F}_5$, we can easily see that $0$ has period type $(0,1)$ when $c=0$; period type $(0,4)$ when $c=1$ or $4$; and period type $(0,2)$ when $c=2$ or $3$. 
		This means that there are exactly 5 polynomials of the form $f_{3,c}(x)\in\mathbb{F}_5[x]$ such that $0$ is periodic. By Corollary \ref{one-to-one}, there are exactly $5$ polynomials of the form $f_{3,c}(x)\in\Z_5[x]$ such that $0$ is periodic.
	Since $0$ is periodic modulo $5$ for every $c$ and $5>3=d$, Theorem \ref{last_periodic} applies to every $c\in\Z_5$. Consequently,  there are exactly $5$ PCF polynomials in $\Z_5[x]$ of the form $f_{3,c}(x)$, none of which is strictly preperiodic.
	\end{Example}
    
	\section{Bounds on the number of primitive prime divisors}
	\label{sec:bounds}

In this section, we give an elementary upper bound on the count of primitive prime divisors in the critical orbits of the polynomials $\f(x)\in\Q[x]$. 
We use the notation of \cite{Holly_Krieger}. If $\f(x)\in\Q[x]$, we set $c=a/b$ with $\gcd(a,b)=1$. Then $\f^n(0)=\frac{a_n}{b^{d^{n-1}}}$ where $a_1:=a,a_n, n\ge 2,b$ are integers with $\gcd(a_n,b)=1$. Let $\varrho_d(n,c)$ be the number of primitive prime divisors of $a_n$ and $\omega(a_n)$ be the number of distinct prime divisors of $a_n$. Recall that the height function $h:\Q\to \Z^+$ is defined by $h(a/b)=\max(|a|,|b|)$ when $\gcd(a,b)=1$.
	\begin{Theorem}
		\label{bounding theorem}
		Let $\f(x)\in\Q[x]$ be a polynomial with infinite critical orbit. Then, 
            \[\varrho_d(n,c)\leq\begin{cases}
			d^{n-1}\log_2|a_1| & c\leq -2\text{ and $d$ is even} \\
			d^{n-1}(3+\log_2|b|)-1 &-2< c< -2^{\frac{1}{d-1}} \text{ and $d$ is even} \\
			(d^{n-1}-1)\log_2|b| + \log_2|a_1| & -2^{\frac{1}{d-1}}< c< 0 \text{ and $d$ is even} \\
			(d^{n-1}-1)(\frac{1}{d-1}+\log_2|b|) + \log_2|a_1| & 0<c<1 \text{; or $-1<c<0$ and $d$ is odd}\\
			d^{n-1}(\frac{1}{d-1}+\log_2|a_1|) - \frac{1}{d-1} & c\geq 1 \text{; or $c\leq-1$ and $d$ is odd}
		\end{cases}\]
		In general, $\varrho_d(n,c)\leq d^{n-1}(3+\log_2 h\left(\frac{a_1}{b}\right))+\log_2|a_1|$.
	\end{Theorem}	
		\begin{Proof}
		It is clear that $\varrho_d(n,c)\leq \omega(a_n)\leq \log_2|a_n|$ for $a_n\neq0$. So, it is sufficient to bound $\log_2|a_n|$.
        
        \noindent\textbf{Claim:} For $c\leq-2$ and $d$ is even, $\log_2|c|\leq \log_2|\f^n(0)|\leq d^{n-1}\log_2|c|$.

        \noindent\textbf{Proof of claim:} The proof is by induction. The statement is trivial for $n=1$.
		We assume that the statement holds for $n$. Now, we have
		$|\f^n(0)|^d\geq|c|^d\geq|c|=-c$, or $|\f^n(0)|^d+c\geq 0.$
		This implies that
		\[\log_2|\f^{n+1}(0)|= \log_2|\f^{n}(0)^d + c|\leq \log_2|\f^{n}(0)|^d=d \log_2|\f^{n}(0)|\leq d\cdot d^{n-1} \log_2|c|=d^{n} \log_2|c|.
		\]
		Since $c\leq -2$, we obtain that $\log_2|c^{d-1}+1|= \log_2(|c|^{d-1}-1)\geq \log_2(2^{d-1}-1)\geq 0$. Therefore,
		\[\log_2|\f^{n+1}(0)|= \log_2|\f^{n}(0)^d + c|\geq\log_2|c^d+c|=\log_2|c(c^{d-1}+1)|=\log_2|c| + \log_2|c^{d-1}+1|\geq \log_2|c|.\]
        
        The claim now gives that for $c\leq-2$ and $d$ is even,
		$\log_2|a_n|-d^{n-1}\log_2|b|=\log_2|\f^n(0)|\leq d^{n-1}\log_2|c|$. In other words,
		$\log_2|a_n|\leq d^{n-1}\log_2|a_1|.$
		
		For $-2<c<-2^{\frac{1}{d-1}}$ and $d$ is even, we use the following inequality that can be found in \cite[p. 5519]{Holly_Krieger}
		\[\log_2|a_n|-d^{n-1}\log_2|b|=\log_2|\f^n(0)|\leq (3d^{n-1}-1).\]
		This implies that
		$\log_2|a_n|\leq d^{n-1}(3+\log_2|b|)-1$.
		
		For $-2^{\frac{1}{d-1}}< c<0$ and $d$ is even, we use the inequalities of \cite[Lemma 3.1, p. 5518]{Holly_Krieger}, namely, $|\f^n(0)|\leq|c|$. Therefore,
		\[\log_2|a_n|\leq d^{n-1}\log_2|b|+\log_2|a_1|-\log_2|b|=(d^{n-1}-1)\log_2|b|+\log_2|a_1|.\]
		For an odd $d$, if $c<0$, then it is clear that $\f^n(0)=-f_{d,-c}^n(0)$. Thus, we may assume without loss of generality that $c$ is positive.
		
		For $0<c<1$, by \cite[Lemma 5.5]{Holly_Krieger}, we have
		\[\log_2|a_n|\leq d^{n-1}\log_2|b|+\frac{d^{n-1}-1}{d-1}+\log_2|a_1|-\log_2|b|=(d^{n-1}-1)\left(\frac{1}{d-1}+\log_2|b|\right)+\log_2|a_1|.\]
		Finally, when $c\geq 1$, we again use \cite[Lemma 5.5]{Holly_Krieger} to obtain
		\[\log_2|a_n|\leq d^{n-1}\log_2|b|+\frac{d^{n-1}-1}{d-1}+d^{n-1}(\log_2|a_1|-\log_2|b|)=d^{n-1}\left(\frac{1}{d-1}+\log_2|a_1|\right)-\frac{1}{d-1}.\]
	\end{Proof}
	Theorem \ref{bounding theorem} gives rise to a bound that depends on the specific iterate of the specific map $\f^n$. A natural question is then whether there can be a uniform upper bound for a fixed degree $d$ and a specific iterate $n$.

	\begin{Question}
	\label{question4}
	Fix $d\geq 2$ and $n\geq 1$. Is there a uniform bound on the number of primitive prime divisors of $\f^n(0)$, $\varrho_d(n,c)$, that does not depend on the value of $c$?
	\end{Question}
		
For the case $n=1$ and for $n=2\text{ and $d$ is even}$, choosing $c=p_1\cdot \dots\cdot p_r$, respectively $c=p_1\cdot \dots\cdot p_r-1$, implies that $\f^n(0)$ possesses at least $r$ primitive prime divisors. Therefore, the bound must depend on $c$ for $n=1,2$. However, for $n\ge 3$, the answer is not as trivial. Indeed the answer is no for every $n\geq 1$ as shown in Corollary \ref{last corollary}.
	
	\section{Density questions on primes in critical orbits}
\label{link_sec}	
		
In this section, we study the density of the following set 
\[{A_{d,n}^t:=\{\p:\p\text{ is a primitive prime divisor of } \f^n(0) \text{ and } \nu_{\p}(\f^n(0))=t \text{ for some } c\in K\}},\ t\ge 1.\]
We first note that the density of this set does not depend on $t$. In fact, by Corollary \ref{power_change}, the density of $A_{d,n}^t$ is the same as the density of the set 
\[{A_{d,n}:=\{\p:\p\text{ is a primitive prime divisor of } \f^n(0) \text{ for some } c\in K\}}.\]

This is because if $\p\in A_{d,n}$ and $\disc_c(\G(c))\not\in\p$, then Corollary $\ref{power_change}$ shows that there is an element $c_t\in R$ such that $\nu_{\p}(f_{d,c_t}^n(0))=t$. This in turn implies that $\p\in A_{d,n}^t$. Letting $B_{d,n}$ be the set of primes $\p$ dividing $\disc_c(\G(c))$, we have that $A_{d,n}\setminus B_{d,n}\subseteq A_{d,n}^t\subseteq A_{d,n}$ for any $t\geq 1$. Since the set of primes dividing a certain discriminant is finite, then the densities of the two sets $A_{d,n}^t$ and $A_{d,n}$ in the set of all primes are the same.

As seen in Proposition \ref{prop_prim_per}, we have that $A_{d,n}$ may be described alternatively as \[\{\p:\textrm{ there exists $c\in k_{\p}$ such that the critical orbit of } \f(x)\in k_{\p}[x] \textrm{ is periodic with exact period } n \}.\]

If $0$ is a periodic point of $\f(x)$, then its cycle contains the critical point, so its multiplier is $0$, which is not a root of unity. Hence, by \cite[Theorem 4.5]{dynamicsbook}, $0$ has formal period $n$ if and only if it has exact period $n$. Therefore, we may redefine $A_{d,n}$ as
\[{A_{d,n}=\{\p:\text{ there exists } c\in k_{\p}\text{ such that }\widetilde{\G}(c)=0\}}.\]
It follows that the natural density $\delta(A_{d,n})$ of $A_{d,n}$ is equal to the fraction, $\FPP_{d,n}$, of the elements of the Galois group of $\G(c)$ that fix at least one root, by the Chebotarev density theorem \cite[Ch.~VII, Theorem 13.4]{Neukirch} and \cite[Theorem 9.15]{Frobenius}.
Denoting the splitting field of $\G(c)$ by $\K_{d,n}$, it follows that
\[\delta(A_{d,n})=\FPP_{d,n}:=\frac{\#\{\sigma\in\Gal(\K_{d,n}/K):\sigma\text{ fixes at least one root of }\G(c)\}}{\#\Gal(\K_{d,n}/K)}.\]
\begin{Proposition}
\label{gen_den}
For all $d\geq 2$ and $n\geq 1$, we have
	\[\delta(A_{d,n})>0.\]
\end{Proposition}
\begin{Proof}
Set $D_{d,n}:=\deg_c(G_{d,n}(c))$. Since $\Gal(\K_{d,n}/K)\leq S_{D_{d,n}}$, as it acts faithfully on $D_{d,n}$ roots, where $S_m$ is the symmetric group on $m$ elements, and there is at least one element in $\Gal(\K_{d,n}/K)$
fixing at least one root of $G_{d,n}(c)$, namely the identity element, it follows that 
$\FPP_{d,n} \geq\frac{1}{D_{d,n}!}>0.$
\end{Proof}
We obtain the following consequence.
\begin{Corollary}
\label{infinite_primes}
For all $d\geq 2$ and $n\geq 1$, there are infinitely many primes $\p$ for which there is $c\in K$ such that $\p$ is a primitive prime divisor of $\f^n(0)$.
\end{Corollary}
In fact, if $d=2$ and $K=\Q$, then $\delta(A_{d,n})$ can be explicitly computed under certain conditions.

\begin{Theorem}
\label{deg_2_density}
Let $K=\Q$ and let $\K_{2,n}$ be the splitting field of $G_{2,n}(c)$, $n\ge 1$. If $\Gal(\K_{2,n}/\Q)\cong S_{D_{2,n}}$, then
		\[\delta(A_{2,n}) = \sum_{i=1}^{D_{2,n}}\frac{(-1)^{i+1}}{i!}.\]
		In particular, this identity holds for every $n\le16$, verified computationally.
\end{Theorem}

\begin{Proof}
The proof comes down to finding the proportion of elements of $S_{D_{2,n}}$ that fix at least one element in the set $\{1,\dots,D_{2,n}\}$. Since the number of permutations fixing no element of $\{1,\dots,D_{2,n}\}$ is 
$D_{2,n}!\sum_{i=0}^{D_{2,n}}(-1)^i/i!$, the result follows. 

We used FLINT \cite{FLINT}, together with Dedekind’s theorem on Frobenius elements \cite{Lenstra_Stevenhagen} and Jordan’s theorem \cite[Theorem 13.9]{Wielandt}, to verify that the Galois group of $G_{2,n}(c)$ is isomorphic to $S_{D_{2,n}}$ for all $1\leq n\leq 16$. 
\end{Proof}

	We remark that due to the exponential growth of the degree $D_{2,n}$ of $G_{2,n}(c)$, the computations on the Gleason polynomials grow increasingly time consuming, so we stopped at $n=16$. We also remind the reader that the irreducibility of $G_{2,n}(c)$ over $\Q$ was conjectured in \cite[Conjecture 1.4]{Gleason} for all $n\geq 1$, however, no proof has been established yet.
		
\begin{Remark}
If $\Gal(\K_{2,n}/\Q)\cong S_{D_{2,n}}$ for all $n$ large enough, then
 $$\lim_{n\to\infty}\delta(A_{2,n})=\lim_{n\to\infty}\FPP_{2,n} = 1-\frac{1}{e},\qquad\textrm{and}$$ 
		\[|\delta(A_{2,n})-1+\frac{1}{e}|\leq \frac{1}{(D_{2,n} + 1)!}\leq \frac{1}{(2^{n-2})!}\ \ \ \text{for}\ \ \ n\geq 2.\]
The limit is a direct consequence of the fact that $e^x=\sum_{i=0}^{\infty}  \frac{x^{i}}{i!}$.
The first inequality comes from the error term expression of an alternating series. For the second inequality, the cases $n\leq 4$ can be checked easily. For $n>4$, we have that $n-2>\frac{n}{2}$, so
\begin{align*}
D_{2,n} &=\sum_{m|n}\mu\left(\frac{n}{m}\right) 2^{m-1}\geq 2^{n-1} - \sum_{m|n, \ m\neq n} 2^{m-1}\geq 2^{n-1} - \sum_{m=1}^{ \frac{n}{2}} 2^{m-1}\geq 2^{n-1} - \sum_{m=1}^{n-2} 2^{m-1}.
\end{align*}
It follows that $D_{2,n}\ge   2^{n-2}$, which yields the desired inequality.
\end{Remark}

 We note that the assumption on the Galois group in Theorem \ref{deg_2_density} does not hold in general when $d>2$. In fact, for $d=3$, one can see that $D_{3,3}=8$ giving that $|S_{D_{3,3}}|=8!=40320$. On the other hand, using $\Magma$, one can see that $|\Gal(\K_{3,3}/\Q)|=192$. Moreover, for $d\geq 4$, one can see that $G_{d,2}=c^{d-1}+1=\frac{c^{2(d-1)}-1}{c^{(d-1)}-1}$ splits in the cyclotomic field $\Q(\zeta_{2(d-1)})$ where $\zeta_{2(d-1)}$ is a primitive $2(d-1)$-th root of unity. This field has a Galois group of order $\phi(2(d-1))$ where $\phi$ is the Euler totient function. However, it is easy to see that since $d-1\geq 3$, one has $\phi(2(d-1))\leq (d-1)<2(d-1)\leq(d-1)!=D_{d,2}!$, in particular, $|\Gal(\K_{d,2}/\Q)|<|S_{D_{d,2}}|$.
\section{Polynomials with arbitrarily many primitive prime divisors}
\label{sec:unbounded}

In this section, we tackle Question \ref{question4}. First, we introduce the following generalization of Theorem \ref{first_theorem}.
\begin{Theorem}
	\label{last theorem}
	Let $K$ be a number field. Fix integers $d\ge 2$ and $m\ge 1$ and let $P$ be a set of primes of density zero in the set of all primes. For $1\le i\le m$, let
	\begin{enumerate}
	    \item $n_i$ be distinct positive integers,
	    \item $t_i$ be positive integers,
	    \item $(k_{i,1},k_{i,2},\dots,k_{i,t_i})$ be $t_i$-tuples of positive integers. 
	\end{enumerate}
	Then, there is a set of elements $c\in R$ of positive density such that
	for each $1\leq i\leq m$ and $1\leq j\leq t_i$, there is a primitive prime divisor $\p_{i,j}$ of $f_{d,c}^{n_i}(0)$ with $\nu_{\p_{i,j}}(f_{d,c}^{n_i}(0))={k_{i,j}}$ and $\p_{i,j}\not\in P$. 
\end{Theorem}
\begin{Proof}
	Fix $i$, $1\le i\le m$. As mentioned in Corollary \ref{infinite_primes}, there exist infinitely many primes $\p$ that can appear as primitive prime divisors in the iteration $\f^{n_i}(0)$ for some $c\in K$. Let the set of these primes be $S_{n_i}$, which by Proposition \ref{gen_den} has a density greater than zero. The finite set of primes dividing $\disc_c(G_{d,n_i}(c))$ will be denoted $T_{n_i}$.
	
	Setting $P_{n_i}:=S_{n_i}\setminus (T_{n_i}\cup P)$, we recursively define the following sets 
	\[A_1\subset P_{n_1} \text{ with } |A_1|=t_1,\text{ and } B_1=A_1,\]
	\[A_i\subset P_{n_i}\setminus B_{i-1} \text{ with } |A_i|=t_i,\text{ and } B_i=B_{i-1}\cup A_i,\quad 2\leq i\leq m\]
	It is clear that $B_{m}$ is a set of distinct primes of $K$. 
	Also, $B_m$ is the union of the disjoint sets $A_i$'s. 
	
	Let $A_i=\{\p_{i,j}\}_{j=1}^{t_i}$. For each $\p_{i,j}\in A_i\subset P_{n_i}$, we know that $\p_{i,j}$ is a primitive prime divisor of $f_{d,c_{i,j,0}}^{n_i}(0)$ for some $c_{i,j,0}\in K$.
	Since $\p_{i,j}\not\in T_{n_i}$, it follows that $ \disc_c(G_{d,n_i}(c))\not\in \p_{i,j}$. By Corollary \ref{power_change}, there exists an element $c_{i,j,1}\in R$ such that $\p_{i,j}$ is a primitive prime divisor for $f_{d,c_{i,j,1}}^{n_i}(0)$, and $\nu_{\p_{i,j}}(f_{d,c_{i,j,1}}^{n_i}(0))={k_{i,j}}$. 
	
	We now set $C_i=\{c_{i,j,1}\}_{1\leq j\leq t_i}$ and $C=\underset{i=1}{\overset{m}{\cup}}C_i$.
	For the set of primes $B_m$, the set $C$ satisfies the hypothesis of Theorem \ref{collecting theorem}. Therefore, we get the desired set of $c\in R$. 
\end{Proof}
Theorem \ref{last theorem} gives the following answer to Question \ref{question4}. 
	\begin{Corollary}
		\label{last corollary}
		Fix $d\ge 2$. Let $U=\{(n_i,t_{i})\}_{i=1}^m$ be a finite set of pairs of positive integers. Then there is a set of elements $c\in R$ of positive density such that $f_{d,c}^{n_i}(0)$ has at least $t_{i}$ primitive prime divisors for each $1\leq i\leq m$.
	\end{Corollary}
	Fixing the degree $d$ and the iteration $n$, Corollary \ref{last corollary} implies the existence of a polynomial of the form $\f(x)$, $c\in R$, such that $\f^n(0)$ has arbitrarily many primitive prime divisors. This implies that the upper bound on the number of primitive prime divisors of $\f^n(0)$, see Theorem \ref{bounding theorem}, cannot be independent of $c$.

    The proof of Theorem \ref{last theorem} provides an explicit description of how to find $c$. In what follows, we work out the details of an example.
\begin{Example}
	Fix $K=\Q$, $d=2$ and $m=3$. For $1\leq i\leq 3$, let $n_i$, $t_i$, and $(k_{i,1},k_{i,2},\dots,k_{i,t_i})$ be as follows:
	\begin{flalign*}
	 &n_1=2,&  &t_1=3, & &(k_{1,1},k_{1,2},k_{1,3})=(29,17,5),&\\
	 &n_2=3,&  &t_2=2,&  & (k_{2,1},k_{2,2})=(8,3),&\\
&n_3=4,& & t_3=1,&     & (k_{3,1})=(21).&
	 \end{flalign*}

	We first check the primes dividing the discriminants:
	\[T_2=\emptyset,\ T_3=\{23\},\ \text{and }\ T_4=\{23,2551\}.\]
	We can now find some of the primes in $S_2,S_3,$ and $S_4$. This can be accomplished by using the primes appearing as primitive divisors in the critical orbits of some polynomials. For example,
	\[f_{2,1}^2(0)=2,\ f_{2,1}^3(0)=5,\text{ and }f_{2,1}^4(0)=2\cdot13.\]
	Therefore, we can set $p_{1,1}=2$ with $c_{1,1,0}=1$. Using Corollary \ref{power_change}, we obtain $c_{1,1,1}=2^{29}-1$.
	Continuing in the same manner, we can get
	\[A_1=\{2,3,7\}\ \text{with}\ C_1=\{2^{29}-1,3^{17}-1,7^5-1\},\]
	\[A_2=\{5,19\}\ \text{with}\ C_2=\{326391,4866\},\]
	\[A_3=\{13\}\ \text{with}\ C_3=\{198396633106433791392520\}.\]
	Using Theorem \ref{collecting theorem}, one such choice is 
	 \[f(x):=x^2+24351981847787737533052341852056330671894786203451391,\]
	for which we have 
	\begin{flalign*}
&2^{29}||f^2(0),\ 3^{17}||f^2(0),\ 7^{5}||f^2(0),&
&5^{8}||f^3(0),\ 19^{3}||f^3(0),&
	&13^{21}||f^4(0)&
	\end{flalign*}
	with each of the mentioned primes being a primitive prime divisor for the corresponding iteration.
\end{Example}
\begin{Example}
	In Corollary \ref{last corollary}, we take $K=\Q$, $d=2$ and $U=\{(3,33)\}$. One can check using \Magma that for the polynomial 
	\[f(x):=x^2+13443222075617361812453920142397689133847531746492684885069771,\]
	we have
	\[70321927694409533965768410131069970323274232658951676172460495|f^3(0),\]
	where this divisor is a squarefree number with 33 prime factors. Each of these factors is a primitive prime divisor for $f^3(0)$. In particular, $f^3(0)$ has at least $33$ primitive prime divisors.
\end{Example}
\section{Maximal Galois groups}
\label{last_sec}
Let $g\in K[x]$ be a polynomial dividing an iterate of $\f\in K[x]$. We write $H_n(f,g):=\Gal(K_n/K_{n-1})$ where $K_n$ is the splitting field of $g\circ f^{n}$. One knows that $H_n(f,g)\isom (\Z/d\Z)^t$ for some $0\leq t\leq \deg(g\circ\f^{n-1})$. The group $H_n(f,g)$ is called maximal when $t=\deg(g\circ f^{n-1})$, see \cite{Rafe_jones}. We now refer to the following result. 
\begin{Theorem}{\cite[Theorem 4.3]{Rafe_jones}}
\label{maximal_galois}
Let $d\geq2$ be an integer, $K$ a global field of characteristic not dividing $d$, $\f(x) = x^d + c \in K[x]$, and $g(x) \in K[x]$ divide an iterate of $\f$. Suppose that $n\geq 2$ and $g\circ f^{n-1}$ is irreducible over $K$. 
If there exists $\p$ with $\nu_{\p}(g( f^n(0)))$ prime to $d$, $\nu_{\p}(g( f^i(0))) = 0$ for all $1 \leq i \leq n -1$,
and $v_{\p}(d) = 0$, then $H_n(f,g)$ is maximal.
\end{Theorem}

Theorem \ref{last theorem} along with Theorem \ref{maximal_galois} give rise to the following result. We recall that $\phi$ denotes the Euler-totient function.
\begin{Corollary}
	\label{last_last_corollary}
	    Let $m\geq 1$, $d\geq 2$ and $\zeta_d$ be a primitive $d$-th root of unity. Set $q:=[K(\zeta_d):K]$. There exists a set of elements $c\in R$ of positive density such that the splitting field $F_m$ of $f_{d,c}^m(x)$ has Galois group $\Gal(F_{m}/K)$ of order $qd^{\frac{d^{m}-1}{d-1}}$. Moreover, this is the maximal attainable order for $\Gal(F_{m}/K)$.
	\end{Corollary}

	\begin{Proof}
	Using the notation of Theorem \ref{last theorem}, for $1\leq i\leq m$, we choose $n_i=i$, $t_i=1$, $(k_{i,1})=(1)$ and $P=\{\p:\ d\in\p\}$. According to Theorem \ref{last theorem}, there exists a set of elements $c\in R$ of positive density such that for all $1\leq i\leq m$, there is a primitive prime divisor $\p_i$ of $\f^i(0)$ such that $\nu_{\p_i}(\f^i(0))=1$. 

    For $m=1$, we note that $K(\zeta_d)\subseteq F_1$. This is because the roots of $x^d+c$ are of the form $\zeta_d^k\beta$ for some $\beta$ such that $\beta^d=-c$ and $0\leq k<d$. Using the fundamental theorem of Galois theory, ${|\Gal(F_1/K)|=|\Gal(F_1/K(\zeta_d))|\cdot |\Gal(K(\zeta_d)/K)|}$ where $|\Gal(K(\zeta_d)/K)|=q$. Therefore, in order to show the maximality of $|\Gal(F_1/K)|$, it is enough to show the maximality of $|\Gal(F_1/K(\zeta_d))|$ which is easily seen to be equivalent to proving that $\f(x)$ is irreducible over $K(\zeta_d)$. This holds because if $\beta$ is a root of $f(x)$, then $f(x)$ splits completely over $K(\zeta_d)(\beta)$. So, $|\Gal(F_1/K(\zeta_d))|=|\Gal(K(\zeta_d)(\beta)/K(\zeta_d))|\leq d$ with the equality being achieved if and only if $\beta$ has a minimal polynomial of degree $d$ over $K(\zeta_d)$, namely $f(x)$.
	
	Since, by construction, there is a prime $\p_1||f(0)=c$, and $d\not\in \p_1$ by the choice of the set $P$, then $\p_1$ does not ramify in $K(\zeta_d)$. So, there is a prime $\q_1$ lying above $\p_1$ in $K(\zeta_d)$ such that $\nu_{\q_1}(-c)=\nu_{\q_1}(c)=1$, i.e., $-c$ is not an $n$-th power in $K(\zeta_d)$ for any $n> 1$. Then by Eisenstein's criterion, $f(x)$ is irreducible over $K(\zeta_d)$. This concludes that $|\Gal(F_1/K)|= qd$.
    
    For $m>1$, we use Theorem \ref{maximal_galois} with $g(x)=x=\f^0(x)$. By induction on $i$, $2\leq i\leq m$, $\Gal(F_i/F_{i-1})\cong \left(\Z/d\Z\right)^{d^{i-1}}$. The irreducibility of $\f^i(x)$ is achieved due to the fact that $\f^i(x)=\f(\f^{i-1}(x))=\underset{\alpha\text{ is a root of }\f^{i-1}(x)}{\prod}(\f(x)-\alpha)$. 
 The maximality of the order of $\Gal(F_i/F_{i-1})$ implies that $\f(x)-\alpha$ is irreducible over $K[\alpha]$ for any root $\alpha$ of $\f^{i-1}(x)$. This along with the irreducibility of $\f^{i-1}(x)$ implies the irreducibility of $\f^i(x)$ using Capelli's Theorem, \cite[Lemma 2.4]{Jones_capelli}. 
	
	Since, by the induction hypothesis, $\Gal(F_{i-1}/K)$ has order $q\cdot d^{\frac{d^{i-1}-1}{d-1}}$, and $\Gal(F_i/F_{i-1})\cong \left(\Z/d\Z\right)^{d^{i-1}}$, then the fundamental theorem of Galois theory gives that 
    \[{\Gal(F_i/K)/\Gal(F_{i-1}/K)\cong \Gal(F_i/F_{i-1})\cong \left(\Z/d\Z\right)^{d^{i-1}}},\] hence the order of $\Gal(F_i/K)$ follows. The maximality implied by Theorem \ref{maximal_galois} along with the maximality of $\Gal(F_{i-1}/K)$ implied by the induction lead to the maximality of the order of $\Gal(F_i/K)$ for any $i$, $1\le i\le m$.
	\end{Proof}

	\begin{Example}
    	Let $K=\Q$, $d=2$ and $m=28$. The following primes 
    	\begin{eqnarray*}
    	\{p_i\}_{1\leq i\leq 28}:=\{584092155055244972754905772273391320763763846953663,\\
        3, 5, 13, 11, 29, 19, 31, 43, 101, 59, 47, 67, 61, 97, 89,\\
    	83, 107, 113, 149, 137, 127, 173, 191, 197, 181, 223, 157\},
    	\end{eqnarray*} 
    	are primitive prime divisors of $f^i(0)$, where
    	\[f(x)=x^2+1168184310110489945509811544546782641527527693907326.\]
    	Moreover,  $p_i||f^i(0)$. It follows that 
    	$\Gal(F_{28}/\Q)$ has order $2^{2^{28}-1}$.
	\end{Example}

\end{document}